\documentclass{amsart}

\usepackage{todonotes}
\usepackage{amssymb}
\usepackage{amsmath}
\usepackage{latexsym} 
\usepackage{amsfonts}
\usepackage{amsthm}
\usepackage{mathrsfs}
\usepackage{verbatim} 
\usepackage[all,2cell,ps]{xy}
\usepackage[text={6in,9in},centering]{geometry} 
\usepackage{graphicx}
\usepackage{enumitem}
\usepackage{colortbl}
\usepackage{caption, subcaption}
\usepackage{tikz}
\usetikzlibrary{graphs,graphs.standard}
\usepackage[pdftex,colorlinks=true, pdfstartview=FitV, linkcolor=magenta, citecolor=blue, urlcolor=blue]{hyperref}

\newtheorem{theorem}{Theorem}[section]
\newtheorem{thm}[theorem]{Theorem}
\newtheorem{prop}[theorem]{Proposition}
\newtheorem{cor}[theorem]{Corollary}

\newtheorem{lem}[theorem]{Lemma}

\theoremstyle{definition}
\newtheorem{definition}[theorem]{Definition}

\newtheorem{example}[theorem]{Example}

\newcommand{\kk}{\Bbbk}

\newcommand{\ol}{\overline}

\newcommand{\vl}{\;\vert\;}

\newcommand{\n}{\mathfrak{n}}

\newcommand{\lst}[2]{#1_1,#1_2,\dots,#1_{#2}}
\newcommand{\defi}[1]{\textsf{#1}}
\newcommand{\supp}{\text{Supp}} 

\newcommand{\prun}[2]{\mathscr{I}(#1, #2)}
\newcommand{\V}{V\setminus\{v\}}

\DeclareMathOperator{\neighbors}{N}

\tikzstyle myBG=[line width=1pt,opacity=1]
 
\newcommand{\edge}[4]
{
  \draw[white,myBG]  (#1, #2) -- (#3, #4);
  \draw[black,thick] (#1,#2) -- (#3,#4);

  \draw[fill=black] (#1,#2) circle (2pt);
  \draw[fill = black] (#3,#4) circle (2pt);
}

\newcommand{\drawPolarLinewithBG}[2]
{
  \draw[white,myBG]  (#1) -- (#2);
  \draw[black, thick] (#1) -- (#2);
}

\title[Classifying NCI]{Classifying Nearly Complete Intersection Ideals Generated in Degree Two}

\author[C. Miller]{Charlie Miller}
\thanks{The first author was supported by McGowan Family Fund, a summer research award for students at Hamilton College, Clinton, NY}
\email{charles.miller1015@gmail.com}

\author[B. Stone]{Branden Stone}
\email{stonebranden@gmail.com}

\begin{document}

\begin{abstract}
	Nearly complete intersection ideals were introduced in \cite{Boocher} and defines a special class of monomial ideals in a polynomial ring. These ideals were used to give a lower bound of the total sum of betti numbers that appear a minimal free resolution of a monomial ideal. In this note we give a graph theoretic classification of nearly complete intersection ideals generated in degree two. In doing so, we define a novel graph operation (the inversion) that is motivated by the definition of this new class of ideals. 
\end{abstract}

\maketitle

\section{Introduction}

Let $I$ be a homogeneous ideal in a polynomial ring $R$ over a field $\kk$. We denote the rank of the $i$-th free module in a minimal free resolution of $R/I$ as $\beta_i(R/I)$. The long-standing conjecture of Buchsbaum-Eisenbud \cite{buchsbaum} and Horrocks \cite{hartshorne} states that if $I$ has height $c$, then 
\[
	\beta_i(R/I) \geqslant \binom c i.
\]
While the case when $c \geqslant 5$ is still open, the weaker statement
\[
	\sum \beta_i(R/I) \geqslant 2^c,
\]
known as the {\it Total Rank Conjecture}, has been completely solved for arbitrary ideals (with char $k \not= 2$) by M. Walker \cite{Walker2017}. At the same time, a special case of this conjecture was independently shown for monomial ideals by A. Boocher and J. Seiner \cite{Boocher}. In particular, they show that if $I$ is not a complete intersection, then 
\[
	\sum \beta_i(R/I) \geqslant 2^c + 2^{c-1}. 
\]
In order to achieve this lower bound, the authors reduce to a special class of ideals they define as nearly complete intersections (NCI) (see Definition~\ref{def:nci-ideal}).  Our main theorem, Theorem~\ref{thm:main}, gives a complete characterization of NCI ideals generated in degree 2 by examining the associated graph $G$.  For example, a squarefree monomial ideal $I$ generated in degree 2 is \emph{not} a nearly complete intersection if $P_5$ is an induced subgraph of $G$. Section~\ref{sec:prelim} gives the necessary background information. The main classification theorem is proved in Section~\ref{sec:classify} as well as a new graph operation, the {\it inversion} of a vertex (Definition~\ref{def:inversion}), motivated by the definition of this new class of ideals. \bigskip

\paragraph{\bf Acknowledgments} This project is the result of undergraduate summer research supported by Hamilton College. We are thankful to the faculty in the Math department for their encouragement throughout the process. In particular, we are grateful to Courtney Gibbons for her continual guidance and feedback. The first author would also like to thank the organizers and speakers at the Thematic Program in Commutative Algebra and its Interaction with Algebraic Geometry held at Notre Dame in the summer of 2019. This workshop introduced the foundational material relevant to this work. 

\section{Preliminaries}\label{sec:prelim}

Unless otherwise noted, we let $R = \kk[\lst x n]$ be a standard graded
polynomial ring over a field $\kk$ in $n$ variables. Given a monomial ideal $I
\subseteq R$, the support of $I$ (denoted $\supp(I)$) will refer to the set of
variables appearing in at least one minimal monomial generator. The following
fact about the support is helpful throughout the note and we state it without
proof.
\begin{lem}\label{lem:support}
	Let $R=\kk[\lst x n]$ and $I = (\lst m t)\subseteq R$ be a monomial ideal. 
	If $m\in R$ is a monomial such that
	\[
		\supp(m) \cap \left[\bigcup^t_{i=1}\supp(m_i)\right] = \varnothing,
	\]
	then $\ol{m}\in R/I$ is a non-zero divisor.
\end{lem}

Using notation defined in \cite{Boocher}, $I(x=1)$ is the ideal defined by
setting $x = 1$ for some variable $x$ in the support of $I$. As such, $I
\subseteq I(x=1)$. E.g., if $I = (ab,bc,ac) \subseteq \kk[a,b,c]$, then $I(a=1)
= (b,c)$ and $I \subseteq I(a=1)$. The following was defined in \cite{Boocher}
and is the main object of study in this note.
\begin{definition}[\cite{Boocher}]\label{def:nci-ideal}
	A squarefree monomial ideal $I \subseteq R$ is a \defi{nearly complete
	intersection} if
	\begin{enumerate}
		\item it is generated in degree at least two, 
		\item is not a complete intersection, and
		\item for each variable $x$ in the support of $I$, $I(x=1)$ is a
		complete intersection.
	\end{enumerate}
\end{definition}

\noindent For example, let $I = (ab,ac,bc) \subseteq \kk[a,b,c]$. We see that $I$ is
generated in degree $2$ and is not a complete intersection. Further, for each
element of $\supp(I)$, $I(a=1) = (b,c)$, $I(b=1) = (a,c)$, and $I(c=1) = (a,b)$
are complete intersections.  Thus $I$ is a nearly complete intersection.

The main result, Theorem~\ref{thm:main}, completely classifies the NCIs generated in degree two via
their associated graphs. Throughout, a finite graph $G$ is a pair $G = (V(G), E
(G))$ where $V(G) = \{\lst x n\}$ is the set of vertices of $G$, and $E(G)$ is a
collection edges of $G$ consisting of two element subsets of $V(G)$. We will
further assume all graphs are simple, i.e. not allowing loops and multiple edges
between vertices. 

There exists a one-to-one correspondence between finite simple graphs and
monomial ideals generated degree two. In particular, given a graph $G$, the
{\it edge ideal} $I(G)$ is typically defined by
\[
	I(G) = \left( x_ix_j \vl \{x_i,x_j\} \in E(G) \right) \subseteq \kk[\lst x
	n],
\]
where $V(G) = \{\lst x n\}$. For bookkeeping reasons, we slightly modify the
standard definition of edge ideal to allow for singletons, while at the same
time preserving the one-to-one correspondence. In particular, in
this note the \defi{edge ideal of a graph $G$} is defined as
\[
	I(G) = \left( x_ix_j, x_k \vl \{x_i,x_j\} \in E(G), x_k \in V(G) \text{ is a
	singleton} \right) \subseteq \kk[\lst x n]. 
\]
For example, the graph $G$ below corresponds to the edge ideal $I(G) = (ab, ac,
bc, d) \subseteq \kk[a,b,c,d]$.

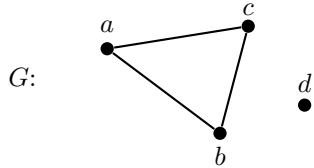
\begin{center}
\begin{tikzpicture}[scale=1.5]

\node (a) at (-.5, .5) [circle,fill=black,inner sep=0pt, minimum size=5pt] {};
\node (b) at (.5, -.25) [circle,fill=black,inner sep=0pt, minimum size=5pt] {};
\node (c) at (.75, .7) [circle,fill=black,inner sep=0pt, minimum size=5pt] {};
\node (d) at (1.25, 0) [circle,fill=black,inner sep=0pt, minimum size=5pt] {};

\draw[black, thick] (a) -- (b) -- (c) -- (a);

\node at (-.5, .7) {$a$};
\node at (.5, -.45) {$b$};
\node at (.75, .85) {$c$};
\node at (1.25, .2) {$d$};
\node at (-1.25, 0.25) {$G$:};
\end{tikzpicture}
\end{center}

Abusing notation we will often refer to an element $u_k \in I(G)$ both as
$u_k = x_{i_k}x_{j_k} \in I(G)$ and $u_k = \{x_{i_k}, x_{j_k}\} \in E(G)$. Using
this correspondence, we say a graph $G$ is a \defi{nearly complete intersection 
(NCI)} if the edge ideal $I(G)$ is a nearly complete intersection. As such
classifying the NCI graphs will in turn classify the NCI ideals generated in degree two. 

We end this section with a standard fact about graphs associated to complete
intersections.

\begin{lem}\label{lem:ci}
	Let $G$ be a simple graph and $I(G)\subseteq R$ be its edge ideal. Then
	$R/I$ is a complete intersection if and only if $G$ is a disjoint union of
	edges and singletons.
\end{lem}

\begin{proof}
	Assume that $E(G) = \{\lst u n\}$ and that $R/I(G)$ is a complete
	intersection. Suppose to the contrary that there exists a vertex $v$ of
	degree 2 in $V(G)$. This implies that there exist
	edges $u_i,u_j \in E(G)$ such that $u_i\cap u_j = \{v\}$. 
	Assuming $i=1$ and $j=2$, we have that $\ol {u_1}\in R/(u_2,\ldots, u_n)$ is
	a zero-divisor. As such there does not exist a vertex of degree two and $G$
	must be a disjoint union of edges and singletons. 

	Assume $G$ is a disjoint union of edges as well as singletons. With out loss of
	generality, we can reduce to the case that $G$ does not contain any
	singletons. Thus, the edge ideal of $G$ is $I(G) = (x_iy_i \vl i =1,\ldots, n)$.  
	Notice that 
		\[
			\supp(x_iy_i)\cap\left[\bigcup^{i-1}_{j=1}\supp(x_jy_j)\right] =\varnothing
		\] 
		for all $i = 1,\ldots,n$. Therefore, $R/I(G)$ is a complete intersection
		by Lemma~\ref{lem:support}.
\end{proof}

\section{Classifying NCIs}\label{sec:classify}

As mentioned in the previous section, a graph $G$ is an NCI if the edge ideal $I
(G)$ is a nearly complete intersection as defined in Definition~\ref{def:nci-ideal}. The main
result, Theorem~\ref{thm:main}, gives a complete classification of NCI ideals
generated in degree at most 2 using the above graph correspondence. Before we can prove the result, we define a new 
graph operation necessary for the proof. We denote the neighbors of a vertex $v$
in $V(G)$, $\neighbors(v)$, and the induced subgraph on a subset $V' \subseteq
V(G)$ as $G[V']$.

\begin{definition}\label{def:inversion}
	The \defi{inversion} of a vertex $v$ in a graph $G$ is the graph defined by 
	\[
		\prun v G = (V',E'),
	\] 
	where $V' = V\setminus\{v\}$ and $E' = E\left(G\left[V'\setminus \mathrm N
	(v)\right]\right)$.
\end{definition}

This operation is a direct translation of the operation $I(x=1)$ used in
Definition~\ref{def:nci-ideal}~(3) and is the main tool used in the
classification of NCI graphs. With it, we can further formalize NCI graphs with
the following lemma whose proof is a direct translation of definitions.

\begin{lem}\label{lem:nci-graph}
	A graph $G$ is an NCI if and only if 
	\begin{enumerate}
		\item $G$ is not a complete intersection, and
		\item for each vertex $v \in V(G)$,	$\prun v G$ is a complete
		intersection. \end{enumerate}
\end{lem}

\noindent From this we have an immediate corollary. 

\begin{cor}
	NCI graphs are connected. 
\end{cor}

This corollary highlights the observations in Section~4 of \cite{Boocher}. In the next example we can use Lemma~\ref{lem:nci-graph} to determine if graphs are NCI or not.

\begin{example} Here we have a graph $G$ and two inversions at the vertices $c$
and $f$. Notice that after the inversions we do not have a complete intersection
(Lemma~\ref{lem:ci}), hence $G$ is not an NCI since every inversion must be a
complete intersection (Lemma~\ref{lem:nci-graph}). 

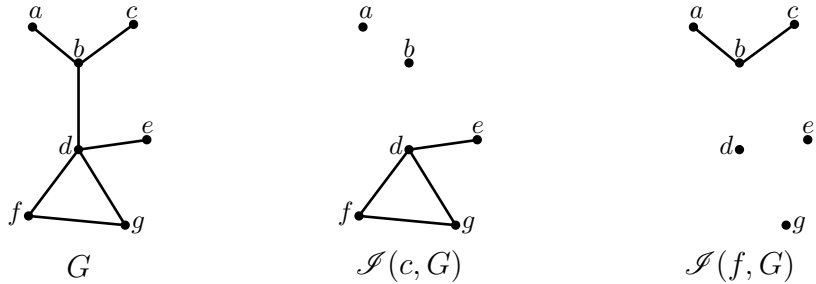
\begin{center}
\tikzset{every picture/.style={line width=1pt}}
\begin{tikzpicture}[x=0.5pt,y=0.5pt,yscale=-1,xscale=1]
 
	\begin{scope}
		\draw    (167,192.03) -- (204.83,141.83) ;
		\draw    (167,192.03) -- (240.13,199) ; 
		\draw    (240.13,199) -- (204.83,141.83) ; 
		\draw    (204.83,141.83) -- (204.83,77.68) ; 
		\draw    (204.83,77.68) -- (246.43,47) ; 
		\draw    (204.83,77.68) -- (170,49) ; 
		\draw    (204.83,141.83) -- (254,134.85) ; 

		\draw  [color={rgb, 255:red, 0; green, 0; blue, 0 }  ,draw opacity=1 ][fill={rgb, 255:red, 0; green, 0; blue, 0 }  ,fill opacity=1 ] (207.33,141.83) .. controls (207.33,140.44) and (206.21,139.33) .. (204.83,139.33) .. controls (203.45,139.33) and (202.33,140.44) .. (202.33,141.83) .. controls (202.33,143.21) and (203.45,144.33) .. (204.83,144.33) .. controls (206.21,144.33) and (207.33,143.21) .. (207.33,141.83) -- cycle ;

		\draw  [color={rgb, 255:red, 0; green, 0; blue, 0 }  ,draw opacity=1 ][fill={rgb, 255:red, 0; green, 0; blue, 0 }  ,fill opacity=1 ] (257.01,136.79) .. controls (258.38,136.49) and (259.25,135.14) .. (258.95,133.78) .. controls (258.65,132.41) and (257.3,131.54) .. (255.94,131.84) .. controls (254.57,132.14) and (253.7,133.49) .. (254,134.85) .. controls (254.3,136.22) and (255.65,137.09) .. (257.01,136.79) -- cycle ;

		\draw  [color={rgb, 255:red, 0; green, 0; blue, 0 }  ,draw opacity=1 ][fill={rgb, 255:red, 0; green, 0; blue, 0 }  ,fill opacity=1 ] (242.63,199) .. controls (242.63,197.62) and (241.51,196.5) .. (240.13,196.5) .. controls (238.75,196.5) and (237.63,197.62) .. (237.63,199) .. controls (237.63,200.38) and (238.75,201.5) .. (240.13,201.5) .. controls (241.51,201.5) and (242.63,200.38) .. (242.63,199) -- cycle ;

		\draw  [color={rgb, 255:red, 0; green, 0; blue, 0 }  ,draw opacity=1 ][fill={rgb, 255:red, 0; green, 0; blue, 0 }  ,fill opacity=1 ] (169.5,192.03) .. controls (169.5,190.65) and (168.38,189.53) .. (167,189.53) .. controls (165.62,189.53) and (164.5,190.65) .. (164.5,192.03) .. controls (164.5,193.41) and (165.62,194.53) .. (167,194.53) .. controls (168.38,194.53) and (169.5,193.41) .. (169.5,192.03) -- cycle ;

		\draw  [color={rgb, 255:red, 0; green, 0; blue, 0 }  ,draw opacity=1 ][fill={rgb, 255:red, 0; green, 0; blue, 0 }  ,fill opacity=1 ] (248.93,47) .. controls (248.93,45.62) and (247.82,44.5) .. (246.43,44.5) .. controls (245.05,44.5) and (243.93,45.62) .. (243.93,47) .. controls (243.93,48.38) and (245.05,49.5) .. (246.43,49.5) .. controls (247.82,49.5) and (248.93,48.38) .. (248.93,47) -- cycle ;

		\draw  [color={rgb, 255:red, 0; green, 0; blue, 0 }  ,draw opacity=1 ][fill={rgb, 255:red, 0; green, 0; blue, 0 }  ,fill opacity=1 ] (172.5,49) .. controls (172.5,47.62) and (171.38,46.5) .. (170,46.5) .. controls (168.62,46.5) and (167.5,47.62) .. (167.5,49) .. controls (167.5,50.38) and (168.62,51.5) .. (170,51.5) .. controls (171.38,51.5) and (172.5,50.38) .. (172.5,49) -- cycle ;

		\draw  [color={rgb, 255:red, 0; green, 0; blue, 0 }  ,draw opacity=1 ][fill={rgb, 255:red, 0; green, 0; blue, 0 }  ,fill opacity=1 ] (207.33,76.18) .. controls (207.33,74.8) and (206.21,73.68) .. (204.83,73.68) .. controls (203.45,73.68) and (202.33,74.8) .. (202.33,76.18) .. controls (202.33,77.56) and (203.45,78.68) .. (204.83,78.68) .. controls (206.21,78.68) and (207.33,77.56) .. (207.33,76.18) -- cycle ;

		\node at (205, 230) {\Large$G$};
		\node at (172.5,38) {$a$};
		\node at (205,65) {$b$};
		\node at (245,38) {$c$}; 
		\node at (195,138) {$d$};
		\node at (257,125) {$e$};
		\node at (157,190) {$f$}; 
		\node at (250,199) {$g$};

	\end{scope}

	\begin{scope}[xshift = 125]

		\draw    (167,192.03) -- (204.83,141.83) ;
		\draw    (167,192.03) -- (240.13,199) ; 
		\draw    (240.13,199) -- (204.83,141.83) ; 
		 
		\draw    (204.83,141.83) -- (254,134.85) ; 

		\draw  [color={rgb, 255:red, 0; green, 0; blue, 0 }  ,draw opacity=1 ][fill={rgb, 255:red, 0; green, 0; blue, 0 }  ,fill opacity=1 ] (207.33,141.83) .. controls (207.33,140.44) and (206.21,139.33) .. (204.83,139.33) .. controls (203.45,139.33) and (202.33,140.44) .. (202.33,141.83) .. controls (202.33,143.21) and (203.45,144.33) .. (204.83,144.33) .. controls (206.21,144.33) and (207.33,143.21) .. (207.33,141.83) -- cycle ;

		\draw  [color={rgb, 255:red, 0; green, 0; blue, 0 }  ,draw opacity=1 ][fill={rgb, 255:red, 0; green, 0; blue, 0 }  ,fill opacity=1 ] (257.01,136.79) .. controls (258.38,136.49) and (259.25,135.14) .. (258.95,133.78) .. controls (258.65,132.41) and (257.3,131.54) .. (255.94,131.84) .. controls (254.57,132.14) and (253.7,133.49) .. (254,134.85) .. controls (254.3,136.22) and (255.65,137.09) .. (257.01,136.79) -- cycle ;

		\draw  [color={rgb, 255:red, 0; green, 0; blue, 0 }  ,draw opacity=1 ][fill={rgb, 255:red, 0; green, 0; blue, 0 }  ,fill opacity=1 ] (242.63,199) .. controls (242.63,197.62) and (241.51,196.5) .. (240.13,196.5) .. controls (238.75,196.5) and (237.63,197.62) .. (237.63,199) .. controls (237.63,200.38) and (238.75,201.5) .. (240.13,201.5) .. controls (241.51,201.5) and (242.63,200.38) .. (242.63,199) -- cycle ;

		\draw  [color={rgb, 255:red, 0; green, 0; blue, 0 }  ,draw opacity=1 ][fill={rgb, 255:red, 0; green, 0; blue, 0 }  ,fill opacity=1 ] (169.5,192.03) .. controls (169.5,190.65) and (168.38,189.53) .. (167,189.53) .. controls (165.62,189.53) and (164.5,190.65) .. (164.5,192.03) .. controls (164.5,193.41) and (165.62,194.53) .. (167,194.53) .. controls (168.38,194.53) and (169.5,193.41) .. (169.5,192.03) -- cycle ;

		\draw  [color={rgb, 255:red, 0; green, 0; blue, 0 }  ,draw opacity=1 ][fill={rgb, 255:red, 0; green, 0; blue, 0 }  ,fill opacity=1 ] (172.5,49) .. controls (172.5,47.62) and (171.38,46.5) .. (170,46.5) .. controls (168.62,46.5) and (167.5,47.62) .. (167.5,49) .. controls (167.5,50.38) and (168.62,51.5) .. (170,51.5) .. controls (171.38,51.5) and (172.5,50.38) .. (172.5,49) -- cycle ;

		\draw  [color={rgb, 255:red, 0; green, 0; blue, 0 }  ,draw opacity=1 ][fill={rgb, 255:red, 0; green, 0; blue, 0 }  ,fill opacity=1 ] (207.33,76.18) .. controls (207.33,74.8) and (206.21,73.68) .. (204.83,73.68) .. controls (203.45,73.68) and (202.33,74.8) .. (202.33,76.18) .. controls (202.33,77.56) and (203.45,78.68) .. (204.83,78.68) .. controls (206.21,78.68) and (207.33,77.56) .. (207.33,76.18) -- cycle ;

		\node at (205, 230) {\Large$\prun c G$};
		\node at (172.5,38) {$a$};
		\node at (205,65) {$b$}; 
		\node at (195,138) {$d$};
		\node at (257,125) {$e$};
		\node at (157,190) {$f$}; 
		\node at (250,199) {$g$};

	\end{scope}

	\begin{scope}[xshift=250]
		\draw    (204.83,77.68) -- (246.43,47) ; 
		\draw    (204.83,77.68) -- (170,49) ; 

		\draw  [color={rgb, 255:red, 0; green, 0; blue, 0 }  ,draw opacity=1 ][fill={rgb, 255:red, 0; green, 0; blue, 0 }  ,fill opacity=1 ] (207.33,141.83) .. controls (207.33,140.44) and (206.21,139.33) .. (204.83,139.33) .. controls (203.45,139.33) and (202.33,140.44) .. (202.33,141.83) .. controls (202.33,143.21) and (203.45,144.33) .. (204.83,144.33) .. controls (206.21,144.33) and (207.33,143.21) .. (207.33,141.83) -- cycle ;

		\draw  [color={rgb, 255:red, 0; green, 0; blue, 0 }  ,draw opacity=1 ][fill={rgb, 255:red, 0; green, 0; blue, 0 }  ,fill opacity=1 ] (257.01,136.79) .. controls (258.38,136.49) and (259.25,135.14) .. (258.95,133.78) .. controls (258.65,132.41) and (257.3,131.54) .. (255.94,131.84) .. controls (254.57,132.14) and (253.7,133.49) .. (254,134.85) .. controls (254.3,136.22) and (255.65,137.09) .. (257.01,136.79) -- cycle ;

		\draw  [color={rgb, 255:red, 0; green, 0; blue, 0 }  ,draw opacity=1 ][fill={rgb, 255:red, 0; green, 0; blue, 0 }  ,fill opacity=1 ] (242.63,199) .. controls (242.63,197.62) and (241.51,196.5) .. (240.13,196.5) .. controls (238.75,196.5) and (237.63,197.62) .. (237.63,199) .. controls (237.63,200.38) and (238.75,201.5) .. (240.13,201.5) .. controls (241.51,201.5) and (242.63,200.38) .. (242.63,199) -- cycle ;

		\draw  [color={rgb, 255:red, 0; green, 0; blue, 0 }  ,draw opacity=1 ][fill={rgb, 255:red, 0; green, 0; blue, 0 }  ,fill opacity=1 ] (248.93,47) .. controls (248.93,45.62) and (247.82,44.5) .. (246.43,44.5) .. controls (245.05,44.5) and (243.93,45.62) .. (243.93,47) .. controls (243.93,48.38) and (245.05,49.5) .. (246.43,49.5) .. controls (247.82,49.5) and (248.93,48.38) .. (248.93,47) -- cycle ;

		\draw  [color={rgb, 255:red, 0; green, 0; blue, 0 }  ,draw opacity=1 ][fill={rgb, 255:red, 0; green, 0; blue, 0 }  ,fill opacity=1 ] (172.5,49) .. controls (172.5,47.62) and (171.38,46.5) .. (170,46.5) .. controls (168.62,46.5) and (167.5,47.62) .. (167.5,49) .. controls (167.5,50.38) and (168.62,51.5) .. (170,51.5) .. controls (171.38,51.5) and (172.5,50.38) .. (172.5,49) -- cycle ;

		\draw  [color={rgb, 255:red, 0; green, 0; blue, 0 }  ,draw opacity=1 ][fill={rgb, 255:red, 0; green, 0; blue, 0 }  ,fill opacity=1 ] (207.33,76.18) .. controls (207.33,74.8) and (206.21,73.68) .. (204.83,73.68) .. controls (203.45,73.68) and (202.33,74.8) .. (202.33,76.18) .. controls (202.33,77.56) and (203.45,78.68) .. (204.83,78.68) .. controls (206.21,78.68) and (207.33,77.56) .. (207.33,76.18) -- cycle ;

		\node at (205, 230) {\Large$\prun f G$};
		\node at (172.5,38) {$a$};
		\node at (205,65) {$b$};
		\node at (245,38) {$c$}; 
		\node at (195,138) {$d$};
		\node at (257,125) {$e$};
		\node at (250,199) {$g$};
	\end{scope}

\end{tikzpicture}
\end{center}

Above shows that not all graphs are NCI. In fact the NCI property seems to be quite
rare. Below are examples of graphs that are NCI. Notice that any inversion of a
vertex will create a disjoint union of edges and singletons, i.e. a complete
intersection. Applying Lemma~\ref{lem:nci-graph} shows they are NCI.

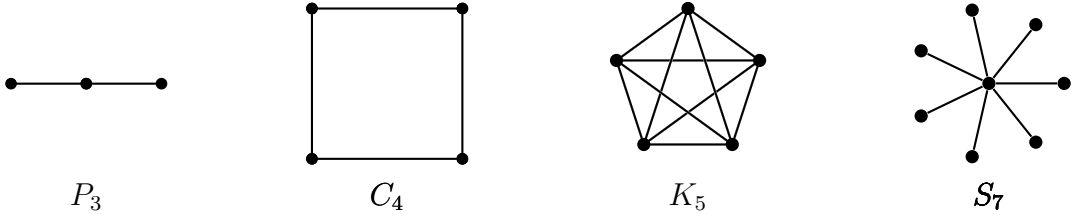
\begin{center}

\begin{tikzpicture}
	\begin{scope}[xshift=1cm]
			\edge {-5} 0 {-4} 0
			\edge {-4} 0 {-3} 0
	\end{scope}

\begin{scope}[yshift= -1cm]
	\edge 0 0 2 0
	\edge 0 2 2 2
	\edge 0 0 0 2
	\edge 2 0 2 2

	\node at (1, -.5) {\Large$C_4$};

\end{scope}

\begin{scope}[xshift=5cm]
    \foreach \a in { 18, 90, 162, 234, 306 } {
      \foreach \b in { 18, 90, 162, 234, 306 } {
        \drawPolarLinewithBG{\a:1}{\b:1};
      }
    }
 
    \foreach \a in {18,90, 162, 234, 306 } {
      \node at (\a:1cm) [circle,fill=black,inner sep=0pt, minimum size=5pt] {};
    }
    \node at (0,-1.5) {\Large$K_5$};
    \node at (-4, -1.5) {\Large$C_4$};
    \node at (-8,-1.5) {\Large$P_3$};
\end{scope}

\begin{scope}[xshift=9cm, yshift=.1]
	\node[circle,fill=black,inner sep=0pt, minimum size=5pt] at (360:0mm) (center) {};
\foreach \n in {1,...,7}{
    \node[circle,fill=black,inner sep=0pt, minimum size=5pt] at ({\n*360/7}:1cm) (n\n) {};
    \draw[black,thick] (center)--(n\n);
    \node at (0,-1.5) {\Large$S_{7}$}; 
}
\end{scope}
\end{tikzpicture}

\end{center}
It's natural to look at the families these graphs belong to. For example, the
family of paths are not all NCI. Indeed, if $n>4$, then the path $P_n$ is not an NCI. To
see this one only needs to invert a leaf of the graph and notice the resulting
graph is not a complete intersection, but another path connecting at least three
vertices. A similar result/argument holds for cycles, i.e. if $n>5$, then a
cycle $C_n$ is not an NCI. However, this is not the case for complete graphs. 
\end{example}

\begin{prop}
	Any complete graph with more than 2 vertices is an NCI.
	\end{prop}
	\begin{proof}
	Let $G = K_n$ be a complete graph on $n\geq 3$ vertices. If $v \in V = V
	(G)$,
	the inversion of $v$ is given by 
	\[
		\prun v G = \left( V', E\left( G \left[ V' \setminus \neighbors(v) \right] \right) \right),
	\] 
	where $V' = \V$. As $G$ is complete, we have that $N(G) = \V = V'$, and
	hence 
	\[
		E\left( G \left[ V' \setminus \neighbors(v) \right] \right) =\emptyset. 
	\]
	Thus $\prun v G$ is a collection of singletons and hence a complete
	intersection by Lemma~\ref{lem:ci}. 
\end{proof} 

In the above path and cycle examples, we saw that the threshold for a graph to be NCI was having $|V(G)| \leqslant 4$ and 5 respectively. It turns out that we can explicitly state the NCI property for connected graphs with at most 4 vertices.

\begin{prop}\label{prop:bound}
	Let $G$ be a connected graph.
	\begin{enumerate}
	 	\item If $|V(G)| \leqslant 2$, then $G$ is not an NCI.
	 	\item If $|V(G)| = 3 \text{ or } 4$, then $G$ is an NCI. 
	 \end{enumerate} 
\end{prop}
\begin{proof}
	When $|V(G)| \leqslant 2$ the graph is a complete intersection by Lemma~\ref{lem:ci} and hence cannot be an NCI by Lemma~\ref{lem:nci-graph}. When $|V(G)| = 3 \text{ or } 4$, $G$ cannot be a complete intersection due to the connected assumption, i.e. any vertex $v \in V(G)$ must be connected to at least one other vertex. Thus $\prun v G$ has at most one edge and is a complete intersection. This forces $G$ to be an NCI. 
\end{proof}

We are now ready to prove the main classification theorem. In the theorem, we
define the graph $T$ as the following. 

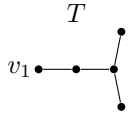
\begin{center}
		\begin{tikzpicture}
			\begin{scope}
				\node (T) at (.5, .75) {$T$};
				\node (v) at (-.25, 0) {$v_1$};
 				\node (a) at (0,0) [circle,fill=black,inner sep=0pt, minimum size=3pt] {};
				\node (b) at (.5,0) [circle,fill=black,inner sep=0pt, minimum size=3pt] {};
				\node (c) at (1,0) [circle,fill=black,inner sep=0pt, minimum size=3pt] {};
				\node (d) at (1.1, .5) [circle,fill=black,inner sep=0pt, minimum size=3pt] {};
				\node (e) at (1.1, -.5) [circle,fill=black,inner sep=0pt, minimum size=3pt] {};

				\draw (a) -- (b) -- (c);
				\draw (d) -- (c) -- (e);
			\end{scope}


		\end{tikzpicture}
\end{center}
This graph, along with $P_5$, become the major obstructions to the NCI property.

\begin{thm}\label{thm:main}
	Let $G$ be a connected graph with $|V(G)|\geqslant 5$. The graph $G$ is not an NCI if and only if there exist vertices $v_1, v_2, v_3, v_4, v_5 \in V(G)$ such that the following conditions hold:\\
	\begin{enumerate}
		\item the vertex $v_1$ is a leaf in $G[v_1, v_2, v_3, v_4, v_5]$;
		\item the path $P_5$ or $T$ is a spanning tree of $G[v_1, v_2, v_3, v_4,
		v_5]$ where the neighbors of $v_1$ all have degree 2 in the spanning tree.
	\end{enumerate}
\end{thm}

\begin{proof}
	Assume $G$ is not an NCI. As such, there exists $v \in V(G)$ such that $\prun v G$ is not a complete intersection. In particular, $\prun v G$ has a vertex $w \in V' = \V$ of degree two. As $G$ is connected, there must exist a path from $v$ to $w$ that passes through the neighbors of $v$ in $G$. So there exists $v_2 \in \neighbors_G(v)$ such that the path 
	\begin{equation}\label{eq:path}
		v \longrightarrow v_2 \longrightarrow v_3 \longrightarrow \cdots \longrightarrow w
	\end{equation}
	exists in $G$. Without losing generality, we can assume the vertices in the path from $v_3$ to $w$ (inclusive) avoid $\neighbors_G(v)$. Indeed if there was a vertex $u \in \neighbors_G(v)$ between $v_3$ and $w$, we could replace $v_2$ with $u$, shortening the path. As such, we may assume the path from $v_3$ to $w$ is completely contained in the subgraph $G[V' \setminus \neighbors_G(v)] \subset G$.  We now consider two cases, $v_3 = w$ and $v_3\not= w$, which can be visualized in the following abstract representation of $G$.

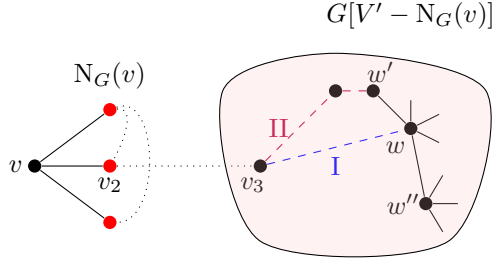
	\begin{center}
	\begin{tikzpicture}
	
		\node (a) at (0,0) [circle,fill=black,inner sep=0pt, minimum size=5pt] {};
		\node (A) at (-.25, 0) {$v$};
		\node (b) at (1,0) [circle,fill=red,inner sep=0pt, minimum size=5pt] {};
		\node (B) at (1,-.25) {$v_2$};
		\node (c) at (1,.75) [circle,fill=red,inner sep=0pt, minimum size=5pt] {};
		\node (d) at (1, -.75) [circle,fill=red,inner sep=0pt, minimum size=5pt] {};

		\node (e) at (3, 0) [circle,fill=black,inner sep=0pt, minimum size=5pt] {};
		\node (E) at (2.9, -.25) {$v_3$};

		\node (g) at (4, 1) [circle,fill=black,inner sep=0pt, minimum size=5pt] {};
		\node (h) at (4.5, 1) [circle,fill=black,inner sep=0pt, minimum size=5pt] {};
		\node (H) at (4.6, 1.25) {$w'$};
		\node (i) at (5, .5) [circle,fill=black,inner sep=0pt, minimum size=5pt] {};
		\node (I) at (4.8, .25) {$w$};
		\node (j) at (5.2, -.5) [circle,fill=black,inner sep=0pt, minimum size=5pt] {};
		\node (J) at (4.9, -.5) {$w''$};		

		\node (k) at (5.5, -1) {};
		\node (o) at (5.5, 0) {};
		\node (p) at (5.75, -.5) {};

		\node (l) at (5.5, .25) {};
		\node (m) at (5.5,.75) {};
		\node (n) at (5, 1)  {};

		\draw (a) -- (b);
		\draw (a) -- (c);
		\draw (a) -- (d);
		\node (xx) at (1, 1.2) {$\neighbors_G(v)$};

		\draw[dotted] (c) to[out=20,in=0] (d);
		\draw[dotted] (c) to[out=5,in=60] (b);
		\draw[dotted] (b) -- (e);
		\draw[dashed, purple] (e) -- node[anchor=east] {II} (g);
		\draw[dashed, purple] (g) -- (h);		
		\draw (h) -- (i) -- (j);

		\draw (j) -- (k);
		\draw (j) -- (o);
		\draw (j) -- (p);

		\draw (i) -- (l);
		\draw (i) -- (m);
		\draw (i) -- (n);

		\draw[dashed, blue] (e) -- node[anchor=north] {I} (i);
		\draw[fill=pink, fill opacity = .2] plot [smooth cycle] coordinates {(2.5,1) (3, -1) (5.5, -1) (6, 1) (4, 1.5)};
		\node (zz) at (5,2) {$G[V'-\neighbors_G(v)]$};
	\end{tikzpicture}
\end{center}

\paragraph{\it Case I} Assume $v_3 = w$. Since $w$ is a degree two vertex in $\prun v G$, there exist $w', w''\in V'$ such that $ww', ww'' \in E' = E\left(G\left[V'\setminus \mathrm \neighbors_G(v)\right]\right)$. In particular $w',w'' \notin\neighbors_G(v)$. As such, $v$ is a leaf in the induced subgraph $H = G[v,v_2,w,w',w'']$, and by construction, $T$ is a spanning tree of $H$ where $v_2$ is the only neighbor of $v$. Further, the degree of $v_2$ is two in the spanning tree $T$, thus both of the desired conditions are satisfied. 

\paragraph{\it Case II} Assume $v_3 \not= w$. As $v_3$ and $w$ are distinct, we can reduce to the case where there is a single vertex between them on the path \eqref{eq:path}, say $w'$. As $v_3, w', w \notin \neighbors_G(v)$, we have that $v$ is a leaf in the induced subgraph $H = G[v, v_2, v_3, w', w]$. In this case, by construction, $P_5$ is a spanning tree of $H$ where $v_2$ is the only neighbor of $v$. As the degree of $v_2$ is two in the spanning tree $P_5$, we have our desired result. 

Conversely, assume the conditions hold for a graph $G$ that is NCI. In this scenario, there exists vertices $v_1, v_2, v_3, v_4, v_5 \in V(G)$ such that $v_1$ is a leaf in the induced subgraph $G[v_1, v_2, v_3, v_4, v_5]$. In the situation where $P_5$ is a spanning tree of $G[v_1, v_2, v_3, v_4, v_5]$, we can assume the vertex labels of the path are as follows.
	
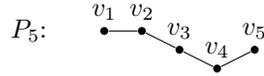
\begin{center}
		\begin{tikzpicture}
			\begin{scope}

				\node (P) at (-1,0) {$P_5$:};
				\node (a) at (0,0) [circle,fill=black,inner sep=0pt, minimum size=3pt] {};
				\node (b) at (0.5,0) [circle,fill=black,inner sep=0pt, minimum size=3pt] {};
				\node (c) at (1, -.25) [circle,fill=black,inner sep=0pt, minimum size=3pt] {};
				\node (d) at (1.5, -.5) [circle,fill=black,inner sep=0pt, minimum size=3pt] {};
				\node (e) at (2, -.25) [circle,fill=black,inner sep=0pt, minimum size=3pt] {};
				\node (aa) at (0,.25) {$v_1$};
				\node (aa) at (0.5,.25) {$v_2$};
				\node (aa) at (1,0) {$v_3$};
				\node (aa) at (1.5,-.25) {$v_4$};
				\node (aa) at (2,0) {$v_5$};																

				\draw (a) -- (b) -- (c) -- (d) -- (e);
			\end{scope}
		\end{tikzpicture}
	\end{center}
Since $v_1$ is a leaf in the induced subgraph, we know that $v_3, v_4, v_5 \notin \neighbors_G(v)$. Hence the degree of $v_4$ is at least two in $\prun{v_1}{G}$. This shows that $\prun{v_1}{G}$ is not a complete intersection, a contradiction of Lemma~\ref{lem:nci-graph}. A similar argument holds for when $T$ is a spanning tree of $G[v_1, v_2, v_3, v_4, v_5]$.
\end{proof}

Theorem~\ref{thm:main}, together with Proposition~\ref{prop:bound} give a complete classification of NCI graphs. As a result, we have a graph theoretic classification of NCI ideals generated in degree two. A natural desire is to extend this result to NCIs with generators in higher degrees. One direction to consider is classifying these ideals with {\it hypergraphs}. A \defi{Hypergraph} is a pair $G = (V,E)$ where $V$ is the set of vertices of $G$ and the set of edges $E$ is a set of nonempty subsets of $V$. In this scenario, more than two vertices can be incident to a single edge. As with graphs, a similar correspondence exits between hypergraphs and ideals and can be seen in the following example.

\begin{example} The left image below is an example of an NCI hypergraph $G$ on a vertex set $V(G) = \{a, b, c, d, e, f, g\}$. Notice this hypergraph has three edges, $\{a,b,c\}$, $\{g\}$, and $\{d,e,f\}$. Further, there is a natural correspondence between these hypergraphs and monomial ideals in $\kk[a,b,c,d,e,f,g]$. In particular $I(G)$ is listed below $G$.
	
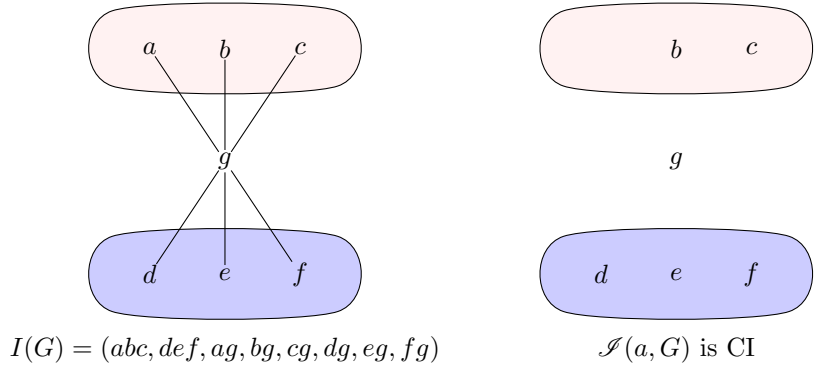
\begin{center}
	\begin{tikzpicture}
	\begin{scope}
		\draw[fill=pink, fill opacity = .2] plot [smooth cycle] coordinates {(0,0) (3, 0) (3, -1) (0, -1)};


		\draw[fill=blue, fill opacity = .2] plot [smooth cycle] coordinates {(0,-3) (3, -3) (3, -4) (0, -4)};


		\node (a) at (.5, -.5) [circle,inner sep=0pt, minimum size=5pt] {$a$};
		\node (b) at (1.5, -.5) [circle,inner sep=0pt, minimum size=5pt] {$b$};
		\node (c) at (2.5, -.5) [circle,inner sep=0pt, minimum size=5pt] {$c$};			
		\node (g) at (1.5, -2) [circle,inner sep=0pt, minimum size=5pt] {$g$};								
		\node (d) at (.5, -3.5) [circle,inner sep=0pt, minimum size=5pt] {$d$};
		\node (e) at (1.5, -3.5) [circle,inner sep=0pt, minimum size=5pt] {$e$};
		\node (f) at (2.5, -3.5) [circle,inner sep=0pt, minimum size=5pt] {$f$};						
		\draw (a) -- (g);
		\draw (b) -- (g);
		\draw (c) -- (g);

		\draw (d) -- (g);
		\draw (e) -- (g);
		\draw (f) -- (g);

		\node at (1.5,-4.5) {$I(G) = (abc,def,ag,bg,cg,dg,eg,fg)$};
	\end{scope}
	\begin{scope}[xshift=6cm]										
		\draw[fill=pink, fill opacity = .2] plot [smooth cycle] coordinates {(0,0) (3, 0) (3, -1) (0, -1)};


		\draw[fill=blue, fill opacity = .2] plot [smooth cycle] coordinates {(0,-3) (3, -3) (3, -4) (0, -4)};


		\node (b) at (1.5, -.5) [circle,inner sep=0pt, minimum size=5pt] {$b$};
		\node (c) at (2.5, -.5) [circle,inner sep=0pt, minimum size=5pt] {$c$};			
		\node (g) at (1.5, -2) [circle,inner sep=0pt, minimum size=5pt] {$g$};								
		\node (d) at (.5, -3.5) [circle,inner sep=0pt, minimum size=5pt] {$d$};
		\node (e) at (1.5, -3.5) [circle,inner sep=0pt, minimum size=5pt] {$e$};
		\node (f) at (2.5, -3.5) [circle,inner sep=0pt, minimum size=5pt] {$f$};						


		\node at (1.5,-4.5) {$\prun a G$ is CI};		
	\end{scope}	
	\end{tikzpicture}
\end{center}	

Lemma~\ref{lem:nci-graph} can also be extended to this scenario as well as the definition of inversion. Notice that inverting $a$ (or any vertex) will produce the complete intersection on the right. It is worth noting that all the examples of NCI hypergraphs we were able to construct were related to the above example. This hints at the possibility that all higher degree NCI ideals are related to the above hypergraph. 
\end{example}

We end this section with an observation relating to the original result of \cite{Boocher}. Let $I$ be a height $c$ monomial ideal in a polynomial ring $S$ that is not a complete intersection. A. Boocher and J. Seiner show that $\sum \beta_i(S/I) \geqslant 2^c + 2^{c-1}$. In particular, equality holds if and only if the generating function for $\beta_i(S/I)$ is either 
\[
	(1+3t+2t^2)(1+t)^{c-2} \text{ or } (1+5t+5t^2+t^3)(1+t)^{c-3}.
\]
When $c = 2$ or $3$, respectively, the generating functions are defined by ideals with the betti sequence $\{1,3,2\}$ and $\{1,5,5,1\}$, respectively. We are able to retrieve these sequences from the obstructions noted in Theorem~\ref{thm:main}. Indeed, the edge ideal $I(P_5)$ and $I(T)$ both have the betti sequence $\{1,4,4,1\}$. However, if we connect the end points of the path $P_5$ to create a 5-cycle, the betti sequence becomes $\{1,5,5,1\}$. Similarly, removing a leaf of either $P_5$ or $T$ can create the path $P_4$, obtaining the betti sequence $\{1,3,2\}$.


\end{document}